\documentclass[12pt]{article}
\usepackage{amsmath,amssymb,amsthm,amsfonts,mathrsfs}
\usepackage{appendix}
\usepackage{array,enumitem,graphics,xcolor,yfonts,colonequals}
\usepackage{stmaryrd}
\usepackage[alphabetic]{amsrefs}
\usepackage[all,cmtip]{xy}
\usepackage{tikz-cd}
\usepackage[colorlinks,anchorcolor=blue,citecolor=blue,linkcolor=blue,urlcolor=blue,bookmarksopen=true]{hyperref}
\usepackage{comment}
\usepackage{mathtools}
\usepackage[margin=1in]{geometry}

\usepackage{titlesec}
\titleformat{\section}
  {\normalfont\large\bf}{\thesection}{1em}{}
\titleformat{\subsection}[runin]
  {\normalfont\normalsize\bf}{\thesubsection}{1em}{}

\usepackage{fancyhdr}
\fancypagestyle{titlepage}
{
	\fancyhf{}

	\fancyfoot[l]{
	\href{https://mathscinet.ams.org/mathscinet/msc/msc2020.html}
	}
}

\AtBeginDocument{%
	\def\MR#1{}
}

\newcommand{\bP}{\mathbb{P}}    
\newcommand{\bQ}{\mathbb{Q}}    
\newcommand{\bZ}{\mathbb{Z}}    

\newcommand{\cA}{\mathcal{A}}   
\newcommand{\cC}{\mathcal{C}}   

\newcommand{\cM}{\mathcal{M}}   
\newcommand{\cO}{\mathcal{O}}   

\newcommand{\Db}{\mathrm{D^b}}  
\newcommand{\FM}{\mathrm{FM}}   

\newcommand{\mukaiH}{\widetilde{H}} 

\newtheorem*{thm*}{Theorem}
\newtheorem*{prop*}{Proposition}
\newtheorem*{cor*}{Corollary}
\newtheorem{thm}{Theorem}[section]

\newtheorem{lemma}[thm]{Lemma}

\numberwithin{equation}{section}

\theoremstyle{definition}

\begin{document}
\title{Fourier--Mukai numbers of K3 categories of very general special cubic fourfolds}
\author{Yu-Wei Fan
    \and Kuan-Wen Lai}
\date{}

\newcommand{\ContactInfo}{{
\bigskip\footnotesize

\bigskip
\noindent Y.-W.~Fan,
\textsc{Yau Mathematical Sciences Center\\
Tsinghua University\\
Beijing 100084, P. R. China}\par\nopagebreak
\noindent\textsc{Email:} \texttt{ywfan@mail.tsinghua.edu.cn}

\bigskip
\noindent K.-W.~Lai,
\textsc{Department of Smart Computing and Applied Mathematics \\
Tunghai University \\
No.~1727, Sec.~4, Taiwan~Blvd., Xitun~Dist., Taichung~City 407224, Taiwan}
\par\nopagebreak
\smallskip
\noindent\textsc{National Center for Theoretical Science \\
No.~1, Sec.~4, Roosevelt~Rd., Taipei~City 106319, Taiwan}
\par\nopagebreak
\noindent\textsc{Email:} \texttt{kwlai@thu.edu.tw}
}}

\maketitle
\thispagestyle{titlepage}

\begin{abstract}
We give counting formulas for the number of Fourier--Mukai partners of the K3 category of a very general special cubic fourfold.
\end{abstract}


\section{Introduction}
\label{sect:intro}

The number of isomorphism classes of Fourier--Mukai partners of a very general complex algebraic K3 surface was computed by Oguiso \cite{Ogu02}*{Proposition~(1.10)}. The purpose of this paper is to establish similar counting formulas for the K3 category of a very general special cubic fourfold.

Let $X\subseteq\bP^5$ be a cubic fourfold, namely, a smooth complex cubic hypersurface. Then the bounded derived category of coherent sheaves on $X$ admits a semiorthogonal decomposition
$$
    \Db(X) = \left<
        \cA_X, \cO_X, \cO_X(1), \cO_X(2)
    \right>.
$$
The full triangulated subcategory $\cA_X$, called the \emph{K3 category of $X$}, is an example of a non-commutative K3 surface in the sense of \cite{MS19}*{Definition~2.31}. In contrast to its common usage in the literature, we say two cubic fourfolds are \emph{Fourier--Mukai partners}, or \emph{FM-partners} for short, if their K3 categories are equivalent. It is known that the number of FM-partners of a cubic fourfold up to isomorphism is finite \cite{Huy17}*{Theorem~1.1}. Moreover, this number is equal to $1$ if the lattice
$$
    H^{2,2}(X,\bZ)\coloneqq
    H^4(X,\bZ)\cap H^{2}(X,\Omega_X^2)
$$
has rank~$1$ \cite{Huy17}*{Theorem~1.5~(i)}.

A cubic fourfold is called \emph{special} if $H^{2,2}(X,\bZ)$ has rank at least $2$, or equivalently, if there exists a rank~$2$ saturated sublattice
$
    K\subseteq H^{2,2}(X,\bZ)
$
containing the square of the hyperplane class $h\coloneqq c_1(\cO_X(1))$. In the moduli space of cubic fourfolds, special members marked by such a sublattice $K$ with $\mathrm{disc}(K) = d$ form an irreducible divisor $\cC_d$, which is nonempty if and only if $d\geq 8$ and $d\equiv 0,2\pmod{6}$ \cite{Has00}*{Theorem~1.0.1}. We say a special cubic fourfold $X\in\cC_d$ is \emph{very general} if it is away from a union of countably many divisors. Note that $H^{2,2}(X,\bZ)$ has rank exactly $2$ for such an $X$.

To state the main theorem, let us first consider the ring $\bZ_{2d}$ of integers modulo $2d$ and denote the subset of square roots of unity as
$$
    \left(
        \bZ_{2d}^\times
    \right)_2 \coloneqq \left\{
        n\in\bZ_{2d}^\times
    \;\middle|\;
        n^2\equiv 1\pmod{2d}
    \right\}.
$$
Using the Chinese remainder theorem, one can verify that
$$
    \left|\left(
        \bZ_{2d}^\times
    \right)_2\right| = \begin{cases}
        4 & \text{if}
        \quad d = 2^{a+1} \\
        2^{k+1} & \text{if}
        \quad d = 2p_1^{e_1}\cdots p_k^{e_k} \\
        2^{k+2} & \text{if}
        \quad d = 2^{a+1}p_1^{e_1}\cdots p_k^{e_k}
    \end{cases}
$$
where $a,k\geq1$ and every $p_i$ is an odd prime.

\begin{thm}
\label{mainthm-1}
Let $X\in\cC_d$ be a very general member and $\FM(X)$ be the set of isomorphism classes of FM-partners of $X$.
\begin{itemize}
\item If $d\not\equiv 0 \pmod{3}$, then
$
    |\FM(X)| = \frac{1}{4}\left|\left(
        \bZ_{2d}^\times
    \right)_2\right|.
$
\item If $d\equiv 0 \pmod{3}$ and $d\not\equiv 0 \pmod{9}$, then
$
    |\FM(X)| = \frac{1}{8}\left|\left(
        \bZ_{2d}^\times
    \right)_2\right|.
$
\item If $d\equiv 0 \pmod{9}$ and $\frac{d}{18}\equiv1\pmod3$, then
$
    |\FM(X)| = \frac{1}{4}\left|\left(
        \bZ_{2d}^\times
    \right)_2\right|.
$
\item If $d\equiv 0 \pmod{9}$ and $\frac{d}{18}\equiv 2\pmod3$, then
$
    |\FM(X)| = \frac{1}{2}\left|\left(
        \bZ_{2d}^\times
    \right)_2\right|.
$
\item If $d\equiv 0 \pmod{27}$, then
$
    |\FM(X)| = \frac{3}{4}\left|\left(
        \bZ_{2d}^\times
    \right)_2\right|.
$
\end{itemize}
\end{thm}

The special case when $d\not\equiv 0\pmod 9$ was proved in our earlier work \cite{FL23}*{Proposition~2.6}, where the formulas were used to help us find new examples of rational cubic fourfolds. Because of this, we need to treat only the case when $d\equiv 0\pmod 9$ here. This assumption implies that $d\equiv 0\pmod{18}$ and we will mostly work with
$$
    d' \coloneqq \frac{d}{18}
$$
instead of $d$ in this paper. By the theorem, if $d\equiv0\pmod{27}$, then
$$
    |\FM(X)|=3\cdot2^n
$$
for some $n\geq0$; otherwise, $|\FM(X)|$ is a power of $2$. It would be interesting to understand the origin of the factor $3$ from a geometric perspective. We remark that the case when $X$ has an associated K3 surface was studied previously by Pertusi \cite{Per21}*{Theorem~1.1}.

Our proof of the theorem is mainly based on the works of Addington--Thomas \cite{AT14} and Huybrechts \cite{Huy17} as these works turn the original problem into the counting of certain overlattices. In Section~\ref{sect:K3Cat}, we briefly review these background materials and set up necessary notations. In Section~\ref{sect:FM-overlat}, we introduce the overlattices arising naturally from FM-partners and explain their roles in the counting problem. In Section~\ref{sect:count-overlat}, we count the number of those overlattices and finish the proof of the main theorem.

\subsection*{Acknowledgements}
The second author was supported by the ERC Synergy Grant HyperK (ID: 854361) and is currently supported by the NSTC Research Grant (113-2115-M-029-003-MY3). We would also like to thank an anonymous referee for providing helpful suggestions.

\section{Mukai lattices of the K3 categories}
\label{sect:K3Cat}

The topological K-theory $K_{\rm top}(X)$ of a cubic fourfold is a free abelian group equipped with a natural integral bilinear pairing $\chi(\,\cdot\,,\,\cdot\,)$. It contains the subgroup
$$
    K_{\rm top}(\cA_X)
    \coloneqq\{
        \kappa\in K_{\rm top}(X)
    \mid
        \chi([\cO_X(i)],\kappa) = 0
        \;\text{ for }\; i = 0,1,2
    \},
$$
called the \emph{Mukai lattice of $\cA_X$}, and it has a polarized Hodge structure of K3 type. Indeed, the pairing $\chi(\,\cdot\,,\,\cdot\,)$ is symmetric on $K_{\rm top}(\cA_X)$ and so turns it into a lattice. On the other hand, the Mukai vector defines an embedding
$$
    K_{\rm top}(X)\lhook\joinrel\longrightarrow H^*(X,\bQ)
    : E\longmapsto\mathrm{ch}(E)\cdot\sqrt{\mathrm{td}(X)}
$$
which endows a Hodge structure on $K_{\rm top}(\cA_X)$ by taking the restriction of the Hodge structure on $H^*(X,\bZ)$. By \cite{Huy17}*{Theorem~1.5~(iii)} and \cite{LPZ23}*{Theorem~1.3}, two very general special cubic fourfolds are FM-partners if and only if there exists a Hodge isometry between their Mukai lattices.

Following our previous work \cite{FL23}, we define
$
    \mukaiH(\cA_X,\bZ)
    \coloneqq
    K_{\rm top}(\cA_X)(-1).
$
As an abstract lattice, it is unimodular of signature $(4,20)$, so we have
\begin{equation}
\label{eqn:MukaiLattice}
    \mukaiH(\cA_X,\bZ)
    \cong
    E_8(-1)^{\oplus 2}\oplus U^{\oplus 4}
    \qquad\text{where}\qquad
    U = \begin{pmatrix}
        0 & 1 \\
        1 & 0 
    \end{pmatrix}.
\end{equation}
Let us further define
$
    N(\cA_X)\coloneqq\mukaiH^{1,1}(\cA_X,\bZ)
$
and
$
    T(\cA_X)\coloneqq N(\cA_X)^{\perp\mukaiH(\cA_X,\bZ)}.
$
The projections of the classes $[\cO_{\rm line}(1)]$ and $[\cO_{\rm line}(2)]$ to $K_{\rm top}(\cA_X)$ induce two elements $\lambda_1,\lambda_2\in N(\cA_X)$ which span the sublattice
$$
    A_2(X)\coloneqq\left<
        \lambda_1, \lambda_2
    \right> \cong \begin{pmatrix}
        2 & -1 \\
        -1 & 2
    \end{pmatrix}
    \subseteq N(\cA_X).
$$
By \cite{AT14}*{Proposition~2.3}, there is a commutative diagram of Hodge structures
$$\xymatrix{
    A_2(X)^{\perp\mukaiH(\cA_X,\bZ)}
    \ar[r]^-\sim &
    \left<
        h^2
    \right>^{\perp H^4(X,\bZ)(-1)}
    = H^4(X,\bZ)_{\rm prim}(-1) \\
    T(\cA_X)
    \ar@{^(->}[u]\ar[r]^-\sim  &
    H^{2,2}(X,\bZ)^{\perp H^4(X,\bZ)(-1)}
    \ar@{^(->}[u]
}$$
where the horizontal maps are isomorphisms.

From now on, we fix a very general $X\in\cC_d$ with $d\equiv 0\pmod{9}$. Let us denote for short that
$
    T\coloneqq T(\cA_X)
$
and
$
    S \coloneqq N(\cA_X)\cap A_2(X)^{\perp N(\cA_X)}.
$
Then a direct computation using the fact that $d\equiv 0\pmod{6}$ gives
$$
    H^{2,2}(X,\bZ)\cong
    \begin{pmatrix}
        3 & 0 \\
        0 & 6d'
    \end{pmatrix}
    \qquad\text{whence}\qquad
    N(\cA_X)\cong
    \begin{pmatrix}
        2 & -1 & 0 \\
        -1 & 2 & 0 \\
        0 & 0 & -6d'
    \end{pmatrix}.
$$
Denote by $\{\lambda_1, \lambda_2, \ell\}$ the standard basis for $N(\cA_X)$ so that $\ell^2 = -6d'$. Then
\begin{equation}
\label{eqn:discS}
    S = \left<\ell\right>
    \qquad\text{whence}\qquad
    S^*/S = \left<\frac{\ell}{6d'}\right>
    \cong \bZ_{6d'}.
\end{equation}
Let us also give an explicit formula for the discriminant group of $T$.

\begin{lemma}
\label{lem:discT}
There exists $t_1, t_2\in T$ with $t_1^2 = -6$, $t_2^2 = 6d'$, and $t_1t_2 = 0$ such that
$$
    T^*/T = \left<
        \frac{t_1}{3}
    \right>\oplus\left<
        \frac{t_2}{6d'}
    \right>
    \cong \bZ_3\oplus\bZ_{6d'}.
$$
\end{lemma}

\begin{proof}
As a result of \cite{Nik79}*{Theorem~1.14.4}, we can fix an isomorphism \eqref{eqn:MukaiLattice} such that the basis elements for $N(\cA_X)$ are identified as
$$
    \lambda_1 = e_1 + f_1,
    \qquad
    \lambda_2 = e_2 + f_2 - e_1,
    \qquad
    \ell = e_3 - (3d')f_3
$$
where $\{e_i,f_i\}$, $i=1,2,3$, are the standard bases for the last three copies of $U$. This isomorphism identifies $T$ with the sublattice
$$
    E_8(-1)^{\oplus 2}\oplus U
    \oplus A_2(-1)
    \oplus \bZ(6d')
$$
where
$$
    A_2(-1) = \left<
        e_1 - f_1 - e_2,\,
        e_2 - f_2
    \right>
    \qquad\text{and}\qquad
    \bZ(6d') =
    \left<
        e_3 + (3d')f_3
    \right>.
$$
Take $t_1\coloneqq e_1 - f_1 - 2e_2 + f_2$ and $t_2\coloneqq e_3 + (3d')f_3$ respectively from these factors. Then a direct computation shows that they satisfy the requirements.
\end{proof}

\section{FM-partners and two types of overlattices}
\label{sect:FM-overlat}

Let us retain the setting from the previous section and define $\cM_{S,T}$ to be the set of even overlattices
$
    L\supseteq S\oplus T
$
with $\mathrm{disc}(L) = 3$ such that $S,T\subseteq L$ are both saturated. Our goal is to turn the original counting problem to the counting on the set $\cM_{S,T}$. Let us start by showing that the elements of $\cM_{S,T}$ can be divided into two types.

Using $[S^*:S] = 6d'$ and $[T^*:T] = 18d'$, one can deduce from the chain of inclusions 
$$
    S\oplus T
    \subseteq L
    \subseteq L^*
    \subseteq S^*\oplus T^*
$$
that
$$
    [L\colon S\oplus T] = [S^*\oplus T^*:L^*] = 6d'.
$$
Now, as $S, T\subseteq L$ are saturated, the projections
\begin{equation}
\label{eqn:surj-to-S*/S}
    L^*/(S\oplus T)
    \longrightarrow
    S^*/S\cong\bZ_{6d'}
\end{equation}
\begin{equation}
\label{eqn:surj-to-T*/T}
    L^*/(S\oplus T)
    \longrightarrow
     T^*/T\cong\bZ_3\oplus\bZ_{6d'}
\end{equation}
are both surjective. In particular, the second map is an isomorphism as the orders of the two groups are both $18d'$. Hence, there exist integers $b_1$ and $b_2$ such that
\begin{equation}
\label{eqn:L*/(S+T)}
    L^*/(S\oplus T) = \left<
        \frac{b_1\ell + t_1}{3}
    \right>\oplus\left<
        \frac{b_2\ell + t_2}{6d'}
    \right>
\end{equation}
where $\ell,t_1,t_2$ are as in \eqref{eqn:discS} and Lemma~\ref{lem:discT}.

The subgroup $L/(S\oplus T)\subseteq L^*/(S\oplus T)$ has index $3$, so its image under \eqref{eqn:surj-to-T*/T} followed by the projection to the factor $\bZ_{6d'}$ is either $\bZ_{2d'}$ or $\bZ_{6d'}$. We can use \eqref{eqn:L*/(S+T)} to deduce an explicit expression for $L/(S\oplus T)$ in each of the two cases as follows:
\begin{enumerate}[label=\textup{(\Roman*)}]
    \item\label{item:split}
    If the image is $\bZ_{2d'}$, then
    $$
        L/(S\oplus T) = \left<
            \frac{b_1\ell + t_1}{3}
        \right>\oplus\left<
            \frac{b_2\ell + t_2}{2d'}
        \right>.
    $$
    In view of the surjectivity of \eqref{eqn:surj-to-S*/S}, we may assume that
    \begin{equation}
    \label{eqn:type-I}
        0\leq b_1 < 3
        \qquad\text{and}\qquad
        0\leq b_2 < 2d'
        \quad\text{with}\quad
        \gcd(b_1,3) = \gcd(b_2, 2d') = 1.
    \end{equation}
    
    \item\label{item:cyclic}
    If the image is $\bZ_{6d'}$, then
    $$
        L/(S\oplus T) = \left<
            \frac{kt_1}{3} + \frac{b_3\ell + t_2}{6d'}
        \right>
        \quad\text{for some}\quad
        0\leq k < 3.
    $$
    Here $b_3 = 2d'kb_1 + b_2$. In view of the surjectivity of \eqref{eqn:surj-to-S*/S}, we may assume that
    \begin{equation}
    \label{eqn:type-II}
        0\leq b_3 < 6d'
        \quad\text{with}\quad
        \gcd(b_3, 6d') = 1.
    \end{equation}
\end{enumerate}

In order to relate the set $\cM_{S,T}$ to the set of FM-partners, let us consider the set $\widetilde{\FM}(X)$ of triples $(Y, \phi, \psi)$ where
\begin{itemize}
\item $Y\in\FM(X)$,
\item
$
    \phi\colon S
    \longrightarrow
    S_Y \coloneqq N(\cA_Y)\cap A_2(Y)^{\perp N(\cA_Y)}
$
    is an isometry of rank~$1$ lattices,
\item 
$
    \psi\colon T
    \longrightarrow
    T_Y\coloneqq T(\cA_Y)
$
is a Hodge isometry.
\end{itemize}

\begin{lemma}
\label{lem:FM-triple-to-FM}
The forgetful map
$$
    \widetilde{\FM}(X)\longrightarrow\FM(X)
    : (Y,\phi,\psi)\longmapsto Y
$$
is $4$-to-$1$. Therefore, we have
$
    |\FM(X)| = \frac{1}{4}|\widetilde{\FM}(X)|.
$
\end{lemma}
\begin{proof}
For each $Y\in\FM(X)$, there are exactly two isometries
$
    \phi\colon S\longrightarrow S_Y
$
which are different by a sign. On the other hand, there exists a Hodge isometry 
$
    \psi\colon T\longrightarrow T_Y
$
due to \cite{Huy17}*{Theorem~1.5~(iii)}. Because the only Hodge isometries on $T$ are $\pm 1$ \cite{Huy16}*{Corollary~3.3.5}, the only Hodge isometries from $T$ to $T_Y$ are $\pm\psi$. This shows that the preimage over $Y\in\FM(X)$ consists of $(Y,\pm\phi,\pm\psi)$, so the statement follows.
\end{proof}

For each $(Y,\phi,\psi)\in\widetilde{\FM}(X)$, one can verify that the pullback
$$
    L_{(Y,\phi,\psi)}\coloneqq
    (\phi\oplus\psi)^*\left(
        A_2(Y)^{\perp\mukaiH(\cA_Y,\bZ)}
    \right)
    \subseteq S^*\oplus T^*
$$
is an element of $\cM_{S,T}$. This defines a map
\begin{equation}
\label{eqn:FM-triple-to-Mst}
    \widetilde{\FM}(X)\longrightarrow\cM_{S,T}
    : (Y,\phi,\psi)\longmapsto L_{(Y,\phi,\psi)}.
\end{equation}
Due to the Torelli theorem \cite{Voi86}, if $(Y,\phi,\psi)$ and $(Y',\phi',\psi')$ have the same image under this map, then $Y\cong Y'$. Hence, the forgetful map in Lemma~\ref{lem:FM-triple-to-FM} factors as
$$\xymatrix{
    \widetilde{\FM}(X)\ar[dr]_-{4:1}\ar[r]^-{L_\bullet}
    & \cM_{S,T}\ar[d] \\
    & \FM(X).
}$$
This shows that every fiber of \eqref{eqn:FM-triple-to-Mst} is contained in a fiber of the forgetful map. Moreover, the map is surjective due to the surjectivity of the period map \cite{Laz10}*{Theorem~1.1}; see the end of \cite{FL23}*{Proof of Lemma~2.7} for the details.

\begin{lemma}
\label{lem:FM-overlat}
The map \eqref{eqn:FM-triple-to-Mst} is $2$-to-$1$. As a result, we have
$
    |\FM(X)| = \frac{1}{2}|\cM_{S,T}|.
$
\end{lemma}

\begin{proof}
It is easy to see that $L_{(Y,\phi,\psi)} = L_{(Y,-\phi,-\psi)}$ for every $(Y,\phi,\psi)\in\widetilde{\FM}(X)$. As a consequence, the fiber $\{(Y,\pm\phi,\pm\psi)\}\subseteq\widetilde{\FM}(X)$ over each $Y\in\FM(X)$ is mapped by \eqref{eqn:FM-triple-to-Mst} as the set
$
    \{
        L_{(Y,\phi,\psi)}, L_{(Y,-\phi,\psi)}
    \}.
$
To prove the statement, it suffices to verify $L_{(Y,\phi,\psi)} \neq L_{(Y,-\phi,\psi)}$, that is, one gets a different lattice after replacing $\ell$ with $-\ell$. 

First suppose that $L\coloneqq L_{(Y,\phi,\psi)}$ satisfies~\ref{item:split}. Assume, to the contrary, that $L$ remains the same after replacing $\ell$ with $-\ell$. This assumption implies that
$$
    \frac{b_1\ell + t_1}{3}
    = \pm\left(
        \frac{- b_1\ell + t_1}{3}
    \right)
$$
within the $\bZ_3$ factor of $L/(S\oplus T)\cong\bZ_3\oplus\bZ_{2d'}$, which gives $\frac{2b_1}{3}\ell = 0$ or $\frac{2}{3}t_1 = 0$. Note that the latter case does not occur. In the former case, we get $b_1\equiv 0\pmod{3}$, but this contradicts the fact that $\gcd(b_1,3)=1$.

Now suppose that $L$ satisfies \ref{item:cyclic}. Assume again to the contrary that $L$ remains the same after replacing $\ell$ with $-\ell$. In this case, we get
$$
    \frac{kt_1}{3} + \frac{b_3\ell + t_2}{6d'}
    = \pm\left(
        \frac{kt_1}{3} + \frac{- b_3\ell + t_2}{6d'}
    \right)
    \;\in\; L/(S\oplus T)\cong\bZ_{6d'},
$$
which implies
$
    \frac{b_3}{3d'}\ell = 0
$
or
$
    \frac{2k}{3}t_1 + \frac{2}{3d'}t_2 = 0.
$
The latter case does not occur. In the former case, we get $b_3\equiv 0\pmod{3d'}$, but this contradicts the fact that $\gcd(b_3, 6d') = 1$.

We have proved the first statement, which gives $|\widetilde{\FM}(X)| = 2|\cM_{S,T}|$. It then follows from Lemma~\ref{lem:FM-triple-to-FM} that
$
    |\FM(X)| 
    = \frac{1}{4}|\widetilde{\FM}(X)|
    = \frac{1}{2}|\cM_{S,T}|.
$
\end{proof}

\section{Counting the number of overlattices}
\label{sect:count-overlat}

It remains to count the number of elements in $\cM_{S,T}$, or more precisely, to count the number of lattices of the forms \ref{item:split} and \ref{item:cyclic} which are even. Let us start by proving a basic property about square roots of unity.

\begin{lemma}
\label{lem:1-mod-4delta}
Let $n$ be a positive integer. Then the number of integers $0\leq b < 2n$ which satisfy $b^2 \equiv 1 \pmod{4n}$ is equal to
$
    \frac{1}{2}\left|
        \left(
            \bZ_{4n}^\times
        \right)_2
    \right|.
$
\end{lemma}
\begin{proof}
For each $b$ satisfying the hypothesis, the integers $b$ and $b + 2n$ represent distinct elements in $\left(\bZ_{4n}^\times\right)_2$. This proves the statement.
\end{proof}

Let us first count even overlattices of Type~\ref{item:split}.

\begin{lemma}
\label{lem:count-type-I}
Elements in $\cM_{S,T}$ of Type~\ref{item:split} exist only if $d'\equiv 2\pmod{3}$. If this condition holds, then there are
$
    \left|\left(
        \bZ_{4d'}^\times
    \right)_2\right|
$
many of them.
\end{lemma}
\begin{proof}
An overlattice of Type~\ref{item:split} has the form
$$
    L = S + T + \left<
            \frac{b_1\ell + t_1}{3}
        \right> + \left<
            \frac{b_2\ell + t_2}{2d'}
        \right>
$$
where $(b_1,b_2)$ are as in \eqref{eqn:type-I}. Such a lattice is even, that is, belongs to $\cM_{S,T}$, if and only if
$$
    \left(
        \frac{b_1\ell + t_1}{3}
    \right)^2
    = -\frac{2}{3}\left(
        b_1^2d' + 1
    \right)
    \;\in\; 2\bZ
    \qquad\Longleftrightarrow\qquad
    d'\equiv 2\pmod{3}
$$
and
$$    
    \left(
        \frac{b_2\ell + t_2}{2d'}
    \right)^2
    = -\frac{3}{2d'}\left(
        b_2^2 - 1
    \right)
    \;\in\; 2\bZ
    \qquad\Longleftrightarrow\qquad
    b_2^2\equiv 1\pmod{4d'}.
$$
Note that the first equivalence follows from the condition $\gcd(b_1,3)=1$ in \eqref{eqn:type-I}.
From here, we see that even overlattices of Type~\ref{item:split} exist only if $d'\equiv 2\pmod{3}$. In this situation, the number of such lattices equals the number of pairs $(b_1,b_2)$ which satisfy
$$
    b_1 \in\{1,2\}
    \qquad\text{and}\qquad
    0 \leq b_2 < 2d'
    \quad\text{with}\quad
    b_2^2\equiv 1\pmod{4d'}.
$$
By Lemma~\ref{lem:1-mod-4delta}, the number of choices for $b_2$ is equal to
$
    \frac{1}{2}\left|
        \left(
            \bZ_{4d'}^\times
        \right)_2
    \right|.
$
Since there are two choices for $b_1$, the number of desired $(b_1,b_2)$ is equal to
$
    \left|
        \left(
            \bZ_{4d'}^\times
        \right)_2
    \right|.
$
\end{proof}

Recall that an overlattice of Type~\ref{item:cyclic} has the form
$$
    L = S + T + \left<
        \frac{b_3\ell + 2d'kt_1 + t_2}{6d'}
    \right>
$$
where $0\leq k < 3$ and $b_3$ is as in \eqref{eqn:type-II}. This lattice is even if and only if
\begin{equation}
\label{eqn:type-II-even}
    \left(
        \frac{b_3\ell + 2d'kt_1 + t_2}{6d'}
    \right)^2
    = -\frac{1}{6d'}\left(
        b_3^2 + 4d'k^2 - 1
    \right)
    \;\in\; 2\bZ.
\end{equation}
Let us count the number of such overlattices case-by-case.

\begin{lemma}
\label{lem:type-II-b3-k=0}
There are
$
    \frac{1}{2}\left|
        \left(
            \bZ_{12d'}^\times
        \right)_2
    \right|
$
many elements in $\cM_{S,T}$ of Type~\ref{item:cyclic} with $k = 0$.
\end{lemma}
\begin{proof}
When $k = 0$, condition~\eqref{eqn:type-II-even} reduces to
$$
    -\frac{1}{6d'}\left(
        b_3^2 - 1
    \right)
    \;\in\; 2\bZ
    \qquad\Longleftrightarrow\qquad
    b_3^2 \equiv 1 \pmod{12d'}.
$$
By Lemma~\ref{lem:1-mod-4delta}, the number of choices for $b_3$ is equal to
$
    \frac{1}{2}\left|
        \left(
            \bZ_{12d'}^\times
        \right)_2
    \right|.
$
\end{proof}

Now we consider the cases when $k=1,2$.

\begin{lemma}
\label{lem:type-II-b3-k-nonzero}
Elements in $\cM_{S,T}$ of Type~\ref{item:cyclic} with $k = 1$ exist only if $d'\equiv 0\pmod{3}$. If this holds, then there are
$
    \frac{1}{2}\left|
        \left(
            \bZ_{4d'}^\times
        \right)_2
    \right|
$
many of them. The same statement holds for $k = 2$.
\end{lemma}
\begin{proof}
When $k = 1$ (resp. $k = 2$), condition~\eqref{eqn:type-II-even} is equivalent to
\begin{equation}
\label{eqn:mod-12d'}
    b_3^2 + 4d' - 1 \equiv 0 \pmod{12d'}.
\end{equation}
Recall from \eqref{eqn:type-II} that $\gcd(b_3,6d')=1$, whence $b_3^2\equiv 1\pmod{3}$, and the above relation modulo $3$ gives $d'\equiv 0\pmod{3}$. Now write 
$$
    b_3 = b_4 + 2d'm
    \quad\text{with}\quad
    0 \leq b_4 < 2d'.
$$
Then \eqref{eqn:type-II} holds if and only if $m\in\{{0,1,2}\}$ and $\gcd(b_4, 2d') = 1$. From \eqref{eqn:mod-12d'}, we get
\begin{equation}
\label{eqn:b3=b4+2d'm}
    b_3^2 + 4d' -1
    \equiv b_4^2 + 4d'(mb_4 + d'm^2 + 1) - 1
    \equiv 0 \pmod{12d'}.
\end{equation}
This relation modulo $4d'$ gives
\begin{equation}
\label{eqn:b4-sq}
    b_4^2\equiv 1\pmod{4d'}.
\end{equation}
We claim that for each such $b_4$, whether $m = 0,1,2$ is uniquely determined. Indeed, if we write $b_4^2 = 1 + 4d'r$ with $r$ an integer, then
$$
    b_4^2 + 4d'(mb_4 + d'm^2 + 1) - 1
    = 4d'(r + mb_4 + d'm^2 + 1).
$$
Inserting this into \eqref{eqn:b3=b4+2d'm} with $d'\equiv 0\pmod{3}$ in mind reduces the relation to
$$
    r + mb_4 + 1 \equiv 0 \pmod{3}.
$$
Then the claim follows as $b_4\not\equiv 0\pmod{3}$. By Lemma~\ref{lem:1-mod-4delta}, the number of $0\leq b_4 < 2d'$ satisfying \eqref{eqn:b4-sq} is equal to
$
    \frac{1}{2}\left|
        \left(
            \bZ_{4d'}^\times
        \right)_2
    \right|.
$
This completes the proof.
\end{proof}

\begin{proof}[Proof of Theorem~\ref{mainthm-1}]
By Lemmas~\ref{lem:count-type-I}, \ref{lem:type-II-b3-k=0}, \ref{lem:type-II-b3-k-nonzero}, the numbers of elements in $\cM_{S,T}$ of different types and values of $d'$ can be organized into a table:
$$\def\arraystretch{1.4}
\begin{array}{|c|c|c|c|}
    \hline
    & d'\equiv 0\pmod{3}
    & d'\equiv 1\pmod{3}
    & d'\equiv 2\pmod{3}
    \\[2pt]
    \hline
    \ref{item:split}
    & 0
    & 0
    & \left|\left(\bZ_{4d'}^\times\right)_2\right|
    \\[2pt]
    \hline
    \ref{item:cyclic}\,\text{ with }\,k=0
    & \frac{1}{2}\left|\left(\bZ_{12d'}^\times\right)_2\right|
    & \frac{1}{2}\left|\left(\bZ_{12d'}^\times\right)_2\right|
    & \frac{1}{2}\left|\left(\bZ_{12d'}^\times\right)_2\right|
    \\[2pt]
    \hline
    \ref{item:cyclic}\,\text{ with }\,k=1
    & \frac{1}{2}\left|\left(\bZ_{4d'}^\times\right)_2\right|
    & 0
    & 0
    \\[2pt]
    \hline
    \ref{item:cyclic}\,\text{ with }\,k=2
    & \frac{1}{2}\left|\left(\bZ_{4d'}^\times\right)_2\right|
    & 0
    & 0
    \\[2pt]
    \hline
    |\cM_{S,T}|
    & \frac{3}{2}\left|\left(\bZ_{4d'}^\times\right)_2\right|
    & \left|\left(\bZ_{4d'}^\times\right)_2\right|
    & 2\left|\left(\bZ_{4d'}^\times\right)_2\right|
    \\[2pt]
    \hline
\end{array}$$
The formulas can then be deduced from Lemma~\ref{lem:FM-overlat} and a direct computation.
\end{proof}

\bigskip
\bibliography{FMCubic}

\begin{bibdiv}
\begin{biblist}

\bib{AT14}{article}{
      author={Addington, Nicolas},
      author={Thomas, Richard},
       title={Hodge theory and derived categories of cubic fourfolds},
        date={2014},
        ISSN={0012-7094},
     journal={Duke Math. J.},
      volume={163},
      number={10},
       pages={1885\ndash 1927},
         url={https://doi.org/10.1215/00127094-2738639},
      review={\MR{3229044}},
}

\bib{FL23}{article}{
      author={Fan, Yu-Wei},
      author={Lai, Kuan-Wen},
       title={{New rational cubic fourfolds arising from Cremona
  transformations}},
        date={2023},
     journal={Algebr. Geom.},
      volume={10},
      number={4},
       pages={432\ndash 460},
}

\bib{Has00}{article}{
      author={Hassett, Brendan},
       title={Special cubic fourfolds},
        date={2000},
        ISSN={0010-437X},
     journal={Compositio Math.},
      volume={120},
      number={1},
       pages={1\ndash 23},
         url={http://dx.doi.org/10.1023/A:1001706324425},
      review={\MR{1738215}},
}

\bib{Huy16}{book}{
      author={Huybrechts, D.},
       title={Lectures on {K}3 surfaces},
      series={Cambridge Studies in Advanced Mathematics},
   publisher={Cambridge University Press, Cambridge},
        date={2016},
      volume={158},
        ISBN={978-1-107-15304-2},
         url={https://doi.org/10.1017/CBO9781316594193},
      review={\MR{3586372}},
}

\bib{Huy17}{article}{
      author={Huybrechts, Daniel},
       title={The {K}3 category of a cubic fourfold},
        date={2017},
        ISSN={0010-437X},
     journal={Compos. Math.},
      volume={153},
      number={3},
       pages={586\ndash 620},
         url={https://doi.org/10.1112/S0010437X16008137},
      review={\MR{3705236}},
}

\bib{Laz10}{article}{
      author={Laza, Radu},
       title={The moduli space of cubic fourfolds via the period map},
        date={2010},
        ISSN={0003-486X},
     journal={Ann. of Math. (2)},
      volume={172},
      number={1},
       pages={673\ndash 711},
         url={https://doi.org/10.4007/annals.2010.172.673},
      review={\MR{2680429}},
}

\bib{LPZ23}{article}{
      author={Li, Chunyi},
      author={Pertusi, Laura},
      author={Zhao, Xiaolei},
       title={{Derived categories of hearts on Kuznetsov components}},
        date={2023},
     journal={J. London Math. Soc.},
}

\bib{MS19}{incollection}{
      author={Macr\`i, Emanuele},
      author={Stellari, Paolo},
       title={Lectures on non-commutative {K}3 surfaces, {B}ridgeland
  stability, and moduli spaces},
        date={2019},
   booktitle={Birational geometry of hypersurfaces},
      editor={Hochenegger, A.},
      editor={Lehn, M.},
      editor={Stellari, P.},
      series={\rm Lecture Notes of the Unione Matematica Italiana},
   publisher={Springer},
       pages={199\ndash 265},
}

\bib{Nik79}{article}{
      author={Nikulin, V.~V.},
       title={Integer symmetric bilinear forms and some of their geometric
  applications},
        date={1979},
        ISSN={0373-2436},
     journal={Izv. Akad. Nauk SSSR Ser. Mat.},
      volume={43},
      number={1},
       pages={111\ndash 177, 238},
      review={\MR{525944}},
}

\bib{Ogu02}{article}{
      author={Oguiso, Keiji},
       title={K3 surfaces via almost-primes},
        date={2002},
        ISSN={1073-2780},
     journal={Math. Res. Lett.},
      volume={9},
      number={1},
       pages={47\ndash 63},
         url={https://doi.org/10.4310/MRL.2002.v9.n1.a4},
      review={\MR{1892313}},
}

\bib{Per21}{article}{
      author={Pertusi, Laura},
       title={Fourier-{M}ukai partners for very general special cubic
  fourfolds},
        date={2021},
        ISSN={1073-2780},
     journal={Math. Res. Lett.},
      volume={28},
      number={1},
       pages={213\ndash 243},
         url={https://doi.org/10.4310/MRL.2021.v28.n1.a9},
      review={\MR{4248001}},
}

\bib{Voi86}{article}{
      author={Voisin, Claire},
       title={Th\'eor\`eme de {T}orelli pour les cubiques de {${\bf P}^5$}},
        date={1986},
        ISSN={0020-9910},
     journal={Invent. Math.},
      volume={86},
      number={3},
       pages={577\ndash 601},
         url={http://dx.doi.org/10.1007/BF01389270},
      review={\MR{860684}},
}

\end{biblist}
\end{bibdiv}
\bibliographystyle{alpha}

\ContactInfo
\end{document}